\documentclass[11pt,a4paper]{amsart}

\usepackage[latin1]{inputenc}
\usepackage[english]{babel}
\usepackage{amsmath,amsfonts,amssymb,mathrsfs}
\usepackage{tikz}
\usepackage[all]{xy}
\usepackage{xcolor}
\usepackage{hyperref} 

\theoremstyle{plain}
\newtheorem{Thm}{Theorem}
\newtheorem{ThmX}{Theorem}

\newtheorem{Prop}[Thm]{Proposition}
\newtheorem{Lemma}[Thm]{Lemma}
\newtheorem{Cor}[Thm]{Corollary}

\theoremstyle{definition}
\newtheorem{Defn}[Thm]{Definition}

\newtheorem*{Def*}{Definition}

\theoremstyle{remark}
\newtheorem*{Remark}{Remark}

\let \seq=\subseteq

\let \0=\emptyset
\let \1=\infty
\let \al=\alpha
\let \be=\beta

\let \De=\Delta
\let \la=\lambda

\let \si=\sigma

\let \f=\varphi
\let \p=\ldots
\let \h=\textrm

\let \cd=\cdots

\let \cov=\vartriangleleft
\let \voc=\vartriangleright
\let \ncov=\ntriangleleft

\let \i=\iota
\let \t=\theta

\newcommand{\WLOGtwo}{without loss of generality}
\newcommand{\Th}{Theorem}

\newcommand{\Co}{Corollary}

\newcommand{\Pro}{Proposition}
\newcommand{\Pros}{Propositions}
\newcommand{\Def}{Definition}

\newcommand{\Exa}{Example}

\newcommand{\Con}{Conjecture}

\newcommand{\I}{I_n}
\newcommand{\Io}{I}
\newcommand{\F}{F_n}
\newcommand{\Fo}{F}
\newcommand{\FA}{F_n^A}
\newcommand{\Fn}{F_n^{A-\{n\}}}
\newcommand{\Fm}{F_n^{\leq a}}  
\newcommand{\Fs}{F_n^{\geq a}}  
\newcommand{\Fi}{F_n^{a_1:a_2}} 
\newcommand{\J}{\mathfrak{I}}
\newcommand{\id}{\mathrm{id}}

\DeclareMathOperator{\inv}{inv}
\DeclareMathOperator{\exc}{exc}
\DeclareMathOperator{\ct}{ct}
\DeclareMathOperator{\ict}{ict}
\DeclareMathOperator{\di}{di}

\newcommand{\Cam}[2]{Cambridge Studies in Advanced Mathe\-matics, vol.~#1, Cambridge University Press, Cambridge, #2}
\newcommand{\Pam}[2]{Pure and Applied Mathe\-matics, vol.~#1, Academic Press, New York, #2}
\newcommand{\Spr}[2]{Graduate Texts in Mathe\-matics, vol.~#1, Springer, New York, #2}

\newcommand{\Adv}[2]{\emph{Adv. Math.} \textbf{#1} (#2)}                   
\newcommand{\Geom}[2]{\emph{Geom. Dedicata} \textbf{#1} (#2)}              
\newcommand{\Algcomb}[2]{\emph{J. Algebraic Combin.} \textbf{#1} (#2)}     
\newcommand{\Proc}[2]{\emph{Proc. Amer. Math. Soc.} \textbf{#1} (#2)}      
\newcommand{\Trams}[2]{\emph{Trans. Amer. Math. Soc.} \textbf{#1} (#2)}    

\title[The Bruhat order on conjugation-invariant sets of involutions]{The Bruhat order on conjugation-invariant sets of involutions in the symmetric group}
\author{Mikael Hansson}
\thanks{}
\address{Department of Mathematics, Link\"oping University, SE-581 83, Link\"oping, Sweden}
\email{mikael.hansson@liu.se}

\begin{document}

\begin{abstract}
Let $\I$ be the set of involutions in the symmetric group $S_n$, and for $A \seq \{0,1,\p,n\}$, let
\[
\FA=\{\si \in \I \mid \h{$\si$ has $a$ fixed points for some $a \in A$}\}.
\]
We give a complete characterisation of the sets $A$ for which $\FA$, with the order induced by the Bruhat order on $S_n$, is a graded poset. In particular, we prove that $\F^{\{1\}}$ (i.e., the set of involutions with exactly one fixed point) is graded, which settles a conjecture of Hultman in the affirmative. When $\FA$ is graded, we give its rank function. We also give a short new proof of the EL-shellability of $\F^{\{0\}}$ (i.e., the set of fixed point-free involutions), which was recently proved by Can, Cherniavsky, and Twelbeck.

\bigskip

\begin{sloppypar}
\noindent \textbf{Keywords:} Bruhat order, symmetric group, involution, conjugacy class, graded poset, EL-shellability
\end{sloppypar}
\end{abstract}

\maketitle


\section{Introduction} \label{intro}

Partially ordered by the Bruhat order, the symmetric group $S_n$ is a graded poset whose rank function is given by the number of inversions, and Edelman~\cite{Edelman} proved that it is EL-shellable. Incitti~\cite{Incitti} proved that the set $\I$ of involutions in $S_n$ is graded with rank function given by the average of the number of inversions and the number of exceedances, and that it is EL-shellable. Hultman~\cite{Hultman3} studied (in a more general setting, which we shall describe shortly) the sets $\F^0$ and $\F^1$ of involutions with no fixed points and with exactly one fixed point, respectively. He proved that $\F^0$ is graded and conjectured that the same is true for $\F^1$. Can, Cherniavsky, and Twelbeck~\cite{C-C-T} recently proved that $\F^0$ is EL-shellable.

We consider the following generalisation. For $a \in \{0,1,\p,n\}$, let $\F^a$ be the conjugacy class in $S_n$ consisting of the involutions with $a$ fixed points, and for $A \seq \{0,1,\p,n\}$, let
\[
\FA=\bigcup_{a \in A}\F^a.
\]
Both $\I$ and $\FA$ are regarded as posets with the order induced by the Bruhat order on $S_n$. Note that
\[
\FA=\{\si \in \I \mid \h{$\si$ has $a$ fixed points for some $a \in A$}\}.
\]
Also note that for all elements in $\I$, the number of fixed points equals $n$ modulo 2. Hence, we may assume that all members of $A$ have the same parity as $n$.

Depicted in Figures~\ref{big} and \ref{med}, are the Hasse diagrams of $\Io_4$, $\Fo_4^0$, and $\Fo_4^2$.

\begin{figure}[htbp]
\begin{center}
\begin{tikzpicture}[xscale=1.5,yscale=1]
  \node at (0,0) {$\diamond$}; \node at (0,4) {$\circ$};
  \node at (-1,1) {$\bullet$}; \node [right] at (0,1) {\tiny $1324$}; \node at (1,1) {$\bullet$};
  \node at (-1,2) {$\bullet$}; \node [left] at (0,2) {\tiny $2143$}; \node at (1,2) {$\bullet$};
  \node at (-0.5,3) {$\circ$}; \node at (0.5,3) {$\bullet$};
  \node [below] at (0,0) {\tiny $1234$}; \node [above] at (0,4) {\tiny $4321$};
  \node [left] at (-1,1) {\tiny $1243$}; \node at (0,1) {$\bullet$}; \node [right] at (1,1) {\tiny $2134$};
  \node [left] at (-1,2) {\tiny $1432$}; \node at (0,2) {$\circ$}; \node [right] at (1,2) {\tiny $3214$};
  \node [left] at (-0.5,3) {\tiny $3412$}; \node [right] at (0.5,3) {\tiny $4231$};
  \draw (0,0)--(1,1)--(1,2)--(0.5,3)--(0,4)--(-0.5,3)--(-1,2)--(-1,1)--(0,0);
  \draw (0,0)--(0,1)--(1,2)--(-0.5,3);
  \draw (0,0)--(0,1)--(-1,2)--(0.5,3);
  \draw (1,1)--(0,2)--(-0.5,3);
  \draw (-1,1)--(0,2)--(0.5,3);
\end{tikzpicture}
\end{center}
\caption{Hasse diagram of $\Io_4$ with the involutions with zero ($\circ$), two ($\bullet$), and four ($\diamond$) fixed points indicated.} \label{big}
\end{figure}
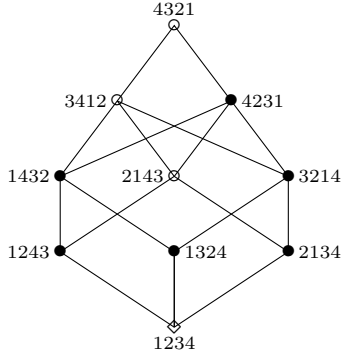

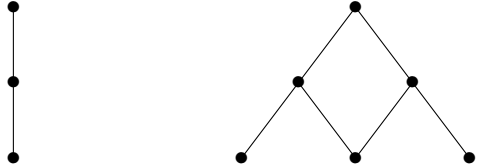
\begin{figure}[htbp]
\begin{center}
\begin{tikzpicture}[xscale=1.5,yscale=1]
  \node at (0,0) {$\bullet$}; \node at (0,1) {$\bullet$}; \node at (0,2) {$\bullet$};
  \draw (0,0)--(0,1)--(0,2);
  \node at (2,0) {$\bullet$}; \node at (3,0) {$\bullet$}; \node at (4,0) {$\bullet$};
  \node at (2.5,1) {$\bullet$}; \node at (3.5,1) {$\bullet$};
  \node at (3,2) {$\bullet$};
  \draw (2,0)--(2.5,1)--(3,2);
  \draw (4,0)--(3.5,1)--(3,2);
  \draw (2.5,1)--(3,0)--(3.5,1);
\end{tikzpicture}
\end{center}
\caption{Hasse diagrams of $\Fo_4^0$ (left) and $\Fo_4^2$ (right).} \label{med}
\end{figure}

\medskip

Our main result is a complete characterisation of the sets $A$ for which $\FA$ is graded. In particular, we prove that $\F^1$ is graded.

Informally, $\FA$ is graded precisely when $A-\{n\}$ is empty or an ``interval,'' which may consist of a single element if it is 0, 1, or $n-2$. The following theorem, which is the main result of this paper, makes the above precise. It also gives the rank function of $\FA$ when it exists.

\begin{Thm} \label{MT}
The poset $\FA$ is graded if and only if $A-\{n\}=\0$ or $A-\{n\}=\{a_1,a_1+2,\p,a_2\}$ with $a_1 \in \{0,1\}$, $a_2=n-2$, or $a_2-a_1 \geq 2$. Furthermore, when $\FA$ is graded, its rank function $\rho$ is given by
\[
\rho(\si)=\frac{\inv(\si)+\exc(\si)-n+\tilde{a}}{2}+\begin{cases}1 & \h{if $n \in A$} \\ 0 & \h{otherwise,}\end{cases}
\]
where $\inv(\si)$ and $\exc(\si)$ denote the number of inversions and exceedances, respectively, of $\si$, and $\tilde{a}=\max(A-\{n\})$. In particular, $\FA$ has rank
\[
\rho(\FA)=\frac{\frac{n-a}{2}(n+a-1)-n+\tilde{a}}{2}+\begin{cases}1 & \h{if $n \in A$} \\ 0 & \h{otherwise,}\end{cases}
\]
where $a=\min{A}$.
\end{Thm}

The following result is direct consequence of \Th~\ref{MT}.

\begin{Cor} \label{gr}
The posets $\F^0$, $\F^1$, $\F^{n-2}$, and $\F^n$ are the only graded conjugacy classes of involutions in $S_n$. Furthermore, the rank function $\rho$ of $\F^0$ and $\F^1$ is given by
\[
\rho(\si)=\frac{\inv(\si)-\lfloor{n/2}\rfloor}{2},
\]
and the rank function $\rho$ of $\F^{n-2}$ is given by
\[
\rho(\si)=\frac{\inv(\si)-1}{2}.
\]
\end{Cor}

It is well known that $\F^{n-2}$ is graded (in fact, it coincides with the root poset of the Weyl group $A_{n-1} \cong S_n$). As was mentioned above, the gradedness of $\F^0$ and $\F^1$ were proved and conjectured, respectively, by Hultman~\cite{Hultman3}. These two posets are special cases of a more general construction from the same paper, which we now describe.\footnote{The results below are taken from \cite{Hultman1,Hultman2,Hultman3}. In general, we do not indicate which results are from which paper. For general Coxeter group terminology and results, see \cite{B-B}.}

\medskip

Given a finitely generated Coxeter system $(W,S)$ and an involutive automorphism $\t$ of $(W,S)$ (i.e., a group automorphism $\t$ of $W$ such that $\t(S)=S$ and $\t^2=\id$), let
\[
\i(\t)=\{\t(w^{-1})w \mid w \in W\}
\]
and
\[
\J(\t)=\{w \in W \mid \t(w)=w^{-1}\}
\]
be the sets of \emph{twisted identities} and \emph{twisted involutions}, respectively. Clearly, $\i(\t) \seq \J(\t) \seq W$. Note that when $\t=\id$, $\i(\t)$ and $\J(\t)$ reduce to the sets of the (ordinary) identity and (ordinary) involutions in $W$. Each subset $X \seq W$ is regarded as a poset with the order induced by the Bruhat order on $W$. When $W$ is the symmetric group $S_n$, there is a unique non-trivial automorphism of $(W,S)$, mapping $s_i=(i,i+1)$ to $s_{n-i}$.

We say that $\t$ has the \emph{no odd flip property} if the order of $s\t(s)$ is even or infinite for all $s \in S$ with $s \neq \t(s)$. If $W$ is finite and irreducible, then $\t$ has the no odd flip property, unless $W$ is of type $A_{2n} \cong S_{2n+1}$ or $I_2(2n+1)$ for some $n \geq 1$, and $\t$ is the unique non-trivial automorphism. The poset $\J(\t)$ is always graded. Furthermore, we have the following result, from which it follows that $\F^0$ is graded, as we shall see.

\begin{ThmX}[\mbox{\cite[\Th~4.6 and \Pro~6.7]{Hultman3}}] \label{Hultman}
If $\t$ has the no odd flip property, then $\i(\t)$ is graded with the same rank function as $\J(\t)$.
\end{ThmX}

If $W$ is finite, it contains a greatest element $w_0$, and $\t(w)=w_0ww_0$ defines an involutive automorphism of $(W,S)$. Since $\i(\t)=\{w_0w^{-1}w_0w \mid w \in W\}$ and $\J(\t)=\{w \in W \mid w_0ww_0=w^{-1}\}$,
\[
w_0 \cdot \i(\t)=\{w^{-1}w_0w \mid w \in W\}
\]
and
\[
w_0 \cdot \J(\t)=\{w_0w \mid w_0ww_0=w^{-1}\}=\{w_0w \mid (w_0w)^2=e\}.
\]
Since left (as well as right) multiplication by $w_0$ is a poset anti-automorphism (i.e., an order-reversing bijection whose inverse is order-reversing), $\i(\t)$ is isomorphic to the dual of $[w_0]$, where $[w_0]$ is the conjugacy class of $w_0$, and $\J(\t)$ is isomorphic to the dual of $I(W)$, where $I(W)$ is the set of involutions in $W$.

When $W$ is the symmetric group $S_n$, this $\t$ is the unique non-trivial automorphism of $(W,S)$, and $I(W)=\I$. For $n$ even, $[w_0]=\F^0$, and for $n$ odd, $[w_0]=\F^1$. Thus it follows from \Th~\ref{Hultman} that $\F^0$ is graded.

It was conjectured by Hultman~\cite[\Con~6.1]{Hultman3} that $\i(\t)$ is graded when $W=A_{2n}$. As we have seen, this is equivalent to $\F^1$ being graded, which is the case (\Co~\ref{gr}). Since $\i(\t)$ is graded whenever $W$ is dihedral, as is easily seen, it therefore follows that $\i(\t)$ is graded whenever $W$ is finite and irreducible. From this, we get the following:

\begin{Thm} \label{finite}
If $W$ is finite, then $\i(\t)$ is graded.
\end{Thm}

\begin{proof}
Let $W$ be a finite Coxeter group and let $\t: W \to W$ be an involutive automorphism. If the Coxeter graph of $W$ consists of two disjoint graphs with vertex sets $S_1$ and $S_2$, the Coxeter groups $W_1$ and $W_2$ generated by $S_1$ and $S_2$, respectively, are isomorphic, and $\t(w)=\t(w_1w_2)=w_2w_1$, where $w_i \in W_i$, then $\t$ has the no odd flip property, whence $\i(\t)$ is graded.\footnote{Alternatively, $\i(\t) \cong W_1$ (see \cite[\Exa~3.2]{Hultman3}), whence $\i(\t)$ is graded.} Other\-wise, we may choose $W_1$ and $W_2$ with $W=W_1 \times W_2$, such that there are two involutive automorphisms $\t_1: W_1 \to W_1$ and $\t_2: W_2 \to W_2$ with $\t(w)=\t_1(w_1)\t_2(w_2)$. In this case, it can be seen that $\i(\t) \cong \i(\t_1) \times \i(\t_2)$, whence, by induction, $\i(\t)$ is graded.
\end{proof}

\medskip

Let us also mention a connection to work by Richardson and Springer~\cite{R-S1,R-S2}, who studied a partially ordered set $V$ of orbits of certain symmetric varieties (depending on, inter alia, a group $G$). They did so by defining an order-preserving function $\f: V \to \J(\t) \seq W$ (where the Weyl group $W$ depends on, inter alia, $G$).

When $W=S_n$, $\J(\t)$ is the image of an injective $\f$ (for details, see \cite[\Exa~10.2]{R-S1}). When $n$ is even, the same is true for $\i(\t)$ (see \cite[\Exa~10.4]{R-S1} or \cite[\Exa~3.1]{Hultman3}). Hence, $\I$ and $\F^0$ are isomorphic to the duals of the images of such functions.

However, these are not the only $\FA$ that occur as the image of a $\f$. To describe these sets, and for later purposes, define
\[
\Fm=\bigcup_{i \geq 0}\F^{a-2i} \quad \h{and} \quad \Fs=\bigcup_{i \geq 0}\F^{a+2i},
\]
and for $a_2=a_1+2m$, where $m$ is a positive integer, let
\[
\Fi=\F^{\geq a_1} \cap \F^{\leq a_2}.
\]

As described in \cite{R-S1}, the image of $\f$ can be read off from the corresponding Satake diagram. It follows from Satake diagrams A~III and A~IV in Helgason~\cite[Table~VI]{Helgason} that for each $a \leq n-2$, $\Fs$ is the image of a $\f$. (From Satake diagrams A~I and A~II, it follows that $\J(\t)$ and $\i(\t)$, respectively, are the images of such functions).

\medskip

The remainder of this paper is organised as follows. In Section~\ref{notation}, we agree on notation and gather the necessary definitions and previous results. Then, in Section~\ref{proof}, we prove the main result of this paper (\Th~\ref{MT}). Finally, in Section~\ref{EL}, we give a short new proof of the following result, which was recently proved by Can, Cherniavsky, and Twelbeck.

\begin{ThmX}[\mbox{\cite[\Th~1]{C-C-T}}] \label{C-C-T}
The poset $\F^0$ is EL-shellable.
\end{ThmX}

\section{Notation and preliminaries} \label{notation}

Poset notation and terminology will follow \cite{Stanley}. In particular, if $P$ is a poset and $x \leq y$ in $P$, then $[x,y]=\{z \in P \mid x \leq z \leq y\}$ and $(x,y)=\{z \in P \mid x<z<y\}$. Furthermore, in a finite poset $P$, $\cov$ denotes the covering relation, a chain $x_0<x_1<\cd<x_k$ is \emph{saturated} if $x_{i-1} \cov x_i$ for all $i \in [n]$, $P$ is \emph{bounded} if it has a minimum (denoted by $\hat0$) and a maximum (denoted by $\hat1$), and $P$ is \emph{graded of rank $n$} if every maximal chain has length $n$. In this case, there is a unique \emph{rank function} $\rho: P \to \{0,1,\p,n\}$ such that $\rho(x)=0$ if $x$ is a minimal element of $P$, and $\rho(y)=\rho(x)+1$ if $x \cov y$ in $P$; $x$ has \emph{rank} $i$ if $\rho(x)=i$. An \emph{$x$-$y$-chain} is a \emph{saturated} chain from $x$ to $y$.

Let $P$ be a finite, bounded, and graded poset. An \emph{edge-labelling} of $P$ is a function $\la: \{(x,y) \in P^2 \mid x \cov y\} \to Q$, where $Q$ is a totally ordered set. If $\la$ is an edge-labelling of $P$ and $x_0 \cov x_1 \cov \cd \cov x_k$ is a saturated chain, let $\la(x_0,x_1,\p,x_k)=(\la(x_0,x_1),\la(x_1,x_2),\p,\la(x_{k-1},x_k))$. The chain is said to be \emph{increasing} if $\la(x_{i-1},x_i) \leq \la(x_i,x_{i+1})$ for all $i \in [k-1]$, and \emph{decreasing} if $\la(x_{i-1},x_i)>\la(x_i,x_{i+1})$ for all $i \in [k-1]$. An edge-labelling $\la$ of $P$ is an \emph{EL-labelling} if, for all $x<y$ in $P$, there is exactly one increasing $x$-$y$-chain, say $x_0 \cov x_1 \cov \cd \cov x_k$, and this chain is \emph{lexicographically minimal}, or \emph{lex-minimal}, among the $x$-$y$-chains in $P$ (i.e., if $y_0 \cov y_1 \cov \cd \cov y_k$ is any other $x$-$y$-chain, then $\la(x_{i-1},x_i)<\la(y_{i-1},y_i)$, where $i=\min\{j \in [k] \mid \la(x_{j-1},x_j) \neq \la(y_{j-1},y_j)\}$; this is known as the \emph{lexicographic order}). If $P$ has an EL-labelling, $P$ is said to be \emph{EL-shellable}. The reason for this is the following result, due to Bj\"orner.

\begin{ThmX}[\mbox{\cite[\Th~2.3]{Bjorner}}]
Let $P$ be a finite, bounded, and graded poset. If $P$ is EL-shellable, then it is shellable (i.e., its order complex $\De(P)$ is shellable).
\end{ThmX}

For $\si \in S_n$ and $(k,l) \in [n]^2$, let $\si[k,l]=|\{i \leq k \mid \si(i) \geq l\}|$. The Bruhat order on $S_n$ may be defined as follows (see, e.g., \cite[\Th~2.1.5]{B-B}):

\begin{Defn} \label{order}
Let $\si,\tau \in S_n$. Then $\si \leq \tau$ if and only if $\si[k,l] \leq \tau[k,l]$ for all $(k,l) \in [n]^2$.
\end{Defn}

Let us turn to involutions in the symmetric group. Here, notation will follow \cite{Incitti}.

Let $\si \in S_n$. A \emph{rise} of $\si$ is a pair $(i,j) \in [n]^2$ such that $i<j$ and $\si(i)<\si(j)$. A rise $(i,j)$ is called \emph{free} if there is no $k \in (i,j)$ such that $\si(k) \in (\si(i),\si(j))$. An \emph{inversion} is a pair $(i,j) \in [n]^2$ such that $i<j$ and $\si(i)>\si(j)$. An element $i \in [n]$ is a \emph{fixed point} of $\si$ if $\si(i)=i$, an \emph{exceedance} if $\si(i)>i$, and a \emph{deficiency} if $\si(i)<i$. Let $\inv(\si)$ and $\exc(\si)$ denote the number of inversions and exceedances, respectively, of $\si$.

Let $\si \in \I$. A free rise is \emph{suitable} if it is an $ff$-rise (Type~1), an $fe$-rise (Type~2), an $ef$-rise (Type~3), a non-crossing $ee$-rise (Type~4), a crossing $ee$-rise (Type~5), or an $ed$-rise (Type~6). Here $fe$, e.g., means that $i$ is a fixed point of $\si$ while $j$ is an exceedance, and an $ee$-rise is \emph{crossing} if $\si(i)<j$ and \emph{non-crossing} otherwise. The following definition is a very important one.

\begin{Defn} \label{ct}
Let $\si \in \I$ and let $(i,j)$ be a suitable rise of $\si$. We define a new involution $\ct_{(i,j)}(\si)$ as follows:

\medskip

If $(i,j)$ is of Type~1, then $\ct_{(i,j)}(\si)=\si(i,j)$.

If $(i,j)$ is of Type~2, then $\ct_{(i,j)}(\si)=\si(i,j,\si(j))$.

If $(i,j)$ is of Type~3, then $\ct_{(i,j)}(\si)=\si(i,j,\si(i))$.

If $(i,j)$ is of Type~4, then $\ct_{(i,j)}(\si)=\si(i,j)(\si(i),\si(j))$.

If $(i,j)$ is of Type~5, then $\ct_{(i,j)}(\si)=\si(i,j,\si(j),\si(i))$.

If $(i,j)$ is of Type~6, then $\ct_{(i,j)}(\si)=\si(i,j)(\si(i),\si(j))$.

\medskip

See \cite[Table~1]{Incitti} for pictures describing the action of $\ct_{(i,j)}$ on the diagram of $\si$.
\end{Defn}

If $\tau=\ct_{(i,j)}(\si)$ for some suitable rise $(i,j)$ of $\si$, let $\la(\si,\tau)=(i,j)$. By Lemma~\ref{cover}, this defines an edge-labelling of $\I$ (with $\{(i,j) \in [n]^2 \mid i<j\}$ totally ordered by the lexicographic order, i.e., $(i_1,j_1)<(i_2,j_2)$ if and only if $i_1<i_2$, or $i_1=i_2$ and $j_1<j_2$). Whenever we consider an edge-labelling of $\I$, it is this one. If $\la(\si,\tau)=(i,j)$, then $(i,j)$ is the \emph{label} on the cover $\si \cov \tau$; $(i,j)$ is a \emph{label} on a chain if it is the label on some cover of the chain.

Let $\tau \in \I$ and let $(i,j)$ be an inversion of $\tau$. If $(i,j)$ is a suitable rise of some $\si \in \I$ and $\ct_{i,j}(\si)=\tau$, then $\si$ is unique, and we write $\si=\ict_{i,j}(\tau)$.

For $\si<\tau$ in $\I$, let $\di(\si,\tau)=\min\{i \in [n] \mid \si(i) \neq \tau(i)\}$.

\medskip

We shall need the following results, due to Incitti:

\begin{Lemma}[\mbox{\cite[\Th~5.1]{Incitti}}] \label{cover}
Let $\si,\tau \in \I$. Then $\si \cov \tau$ in $\I$ if and only if $\tau=\ct_{(i,j)}(\si)$ for some suitable rise $(i,j)$ of $\si$.
\end{Lemma}

\begin{Lemma}[\mbox{\cite[\Th~5.2]{Incitti}}] \label{rank}
The poset $\I$ is graded with rank function $\rho$ given by
\[
\rho(\si)=\frac{\inv(\si)+\exc(\si)}{2}.
\]
\end{Lemma}

\begin{Lemma}[\mbox{\cite[\Th~6.2]{Incitti}}] \label{incr}
Let $\si<\tau$ in $\I$. Then there is exactly one increasing $\si$-$\tau$-chain, and it is lex-minimal.
\end{Lemma}

\begin{Lemma}[\mbox{\cite[\Th~7.3]{Incitti}}] \label{decr}
Let $\si<\tau$ in $\I$. Then there is exactly one decreasing $\si$-$\tau$-chain.
\end{Lemma}

\begin{Remark}
Since $\ct_{(i,j)}(\si)(i)>\ct_{(i,j)}(\si)(j)$, there is also exactly one ``weakly'' decreasing $\si$-$\tau$-chain. This fact is used in Section~\ref{EL}.
\end{Remark}

\section{Proof of the main result} \label{proof}

In this section, we prove a number of lemmas and propositions, from which \Th~\ref{MT} easily follows.

The strategy for proving that a poset $\FA$ is graded is as follows. Given $x<y$ in $\FA$ such that $x \ncov y$ in $\I$, we consider the increasing and the decreasing $\si$-$\tau$-chains in $\I$. We then prove that either the element in the increasing chain that covers $\si$, or the element in the decreasing chain that is covered by $\tau$, has to belong to $\FA$. By Lemma~\ref{graded}, this implies that $\FA$ is graded.

For $\Fm$, this is done by assuming the opposite to be true, and then using \Def~\ref{order} to obtain a contradiction (see Lemma~\ref{Fm}). For $\Fs$, more work is needed (see Lemma~\ref{Fs}). The proof for $\Fi$ is largely a combination of the proofs for $\Fm$ and $\Fs$ (see Lemma~\ref{Fi}).

To prove that a poset $\FA$ is not graded, we consider an interval $[\si,\tau]$, and then construct two $\si$-$\tau$-chains in $\FA$ of different lengths (see \Pros~\ref{-o-} and \ref{o-o}).

Let us first note the following fact:

\begin{Lemma} \label{bij}
For all $n$ and all $A$, $\FA$ is graded if and only if $\Fn$ is graded.
\end{Lemma}

\begin{proof}
This is obvious if $n \notin A$. Otherwise, deleting the identity permutation gives a bijection between maximal chains in $\FA$ of length $k$ and maximal chains in $\Fn$ of length $k-1$.
\end{proof}

In the next two results, we describe the maximal and minimal elements of $\FA$.

\begin{Prop} \label{top}
For all $n$ and all $A$, $\FA$ has a $\hat1$. Furthermore, $\inv(\hat1)=\frac{n-a}{2}(n+a-1)$ and $\exc(\hat1)=\frac{n-a}{2}$, where $a=\min{A}$.
\end{Prop}

\begin{proof}
Let
\[
\tau=n(n-1) \cd (\be+1)(\al+1)(\al+2) \cd \be\al(\al-1) \cd 1
\]
(one line notation), where $\al=\frac{n-a}{2}$ and $\be=\frac{n+a}{2}$; note that $\al+\be=n$. We prove that $\tau$ is a $\hat1$.

In order to obtain a contradiction, assume that there are $\si \in \FA$ and $(k,l) \in [n]^2$ such that $\si[k,l]>\tau[k,l]$; note that $k \in [\al+1,\be-1]$. Then either
\begin{equation} \label{eq1}
|\si([k]) \cap [\al+2,n]|=k
\end{equation}
or
\begin{equation} \label{eq2}
|\si([k]) \cap [k+1,n]| \geq \al+1.
\end{equation}
In either case, $\si$ has at most $k-\al-1$ fixed points in $[k]$, and hence at least $(\be-\al)-(k-\al-1)=\be-k+1$ fixed points in $[k+1,n]$. Thus,
\begin{align*}
|\si([k]) \cap [k+1,n]| &\leq (n-(k+1)+1)-(\be-k+1) \\
                        &= n-\be-1=\al-1,
\end{align*}
which contradicts \eqref{eq2}. Furthermore,
\begin{align*}
|\si([k]) \cap [\al+2,n]| &= |\si([k]) \cap [\al+2,k]|+|\si([k]) \cap [k+1,n]| \\
                          &\leq (k-(\al+2)+1)+(\al-1)=k-2,
\end{align*}
which contradicts \eqref{eq1}. Thus $\tau$ is a $\hat1$.

For the second part, we count the number of inversions $(i,j)$ such that $i \leq \al$, $\al<i \leq \be$, and $i>\be$, respectively, as follows:
\begin{align*}
\inv(\hat1) &= [(n-1)+(n-2)+\cd+(n-\al)]+a\al+\tbinom{\al}{2} \\
            &= \al \cdot \tfrac{2n-\al-1}{2}+a\al+\al \cdot \tfrac{\al-1}{2} \\
            &= \al(n+a-1)=\tfrac{n-a}{2}(n+a-1).
\end{align*}
Clearly, $\exc(\hat1)=\frac{n-a}{2}$.
\end{proof}

\begin{Prop} \label{bot}
For all $n$ and all $A$, all minimal elements of $\FA$ have rank $(n-\max{A})/2$ in $\I$.
\end{Prop}

\begin{proof}
Put $a=\max{A}$. We prove that all minimal elements of $\FA$ have disjoint cycle decompositions consisting of $a$ fixed points and $\frac{n-a}{2}$ adjacent transpositions. By Lemma~\ref{rank}, all such involutions have rank $\frac{n-a}{2}$.

We prove the contrapositive. Thus, let $\tau$ be another involution in $\FA$. Then there is an $i \in [n]$ such that $\tau(i)=j \geq i+2$.

\textbf{Case~1.} For some $i$ there is a $k \in (i,j)$ such that $\tau(k) \in (i,j)$. Choose such an $i$, and let $k$ be minimal. Note that $k$ is either a fixed point or an exceedance. In either case, let $\si=\ict_{i,k}(\tau)$; note that $(i,k)$ is a free rise of $\si$ by the minimality of $k$. Then $\tau \voc \si \in \FA$, whence $\tau$ is not minimal.

\textbf{Case~2.} For no $i$ is there such a $k$. Let $i$ be minimal, and let $\tau(l)=i+1$. Then $l \geq j+1$. Let $\si=\ict_{i,l}(\tau)$; note that $(i,l)$ is a free rise of $\si$ by the assumption of Case~2. Then $\tau \voc \si \in \FA$, whence $\tau$ is not minimal.
\end{proof}

Recall that
\[
\Fm=\bigcup_{i \geq 0}\F^{a-2i}, \quad \Fs=\bigcup_{i \geq 0}\F^{a+2i}, \quad \h{and} \quad \Fi=\F^{\geq a_1} \cap \F^{\leq a_2},
\]
where $a_2=a_1+2m$ for some positive integer $m$. Note that $\Fi$ is not defined for $a_1=a_2$.

The following lemma will eventually allow us to conclude that $\Fm$, $\Fs$, and $\Fi$ are graded.

\begin{Lemma} \label{graded}
If there are no $x<y$ in $\FA$ such that $x \ncov y$ in $\I$ and $(x,y) \cap \FA=\0$, then $\FA$ is graded.
\end{Lemma}

\begin{proof}
Let $C=\si \cov \cd \cov \hat1$ be a maximal chain in $\FA$; note that $\si$ is a minimal element of $\FA$. By the assumption, $C$ is saturated in $\I$. Hence, $\ell(C)=\rho(\hat1)-\rho(\si)=\rho(\hat1)-\frac{n-a}{2}$, where $\rho$ is the rank function of $\I$ and $a=\max{A}$. Thus $\FA$ is graded.
\end{proof}

The next lemma is used in the proofs of Lemmas~\ref{Fm}, \ref{Fs}, and \ref{Fi}, which, together with Lemma~\ref{graded}, show that $\Fm$, $\Fs$, and $\Fi$ are graded.

\begin{Lemma} \label{ineq}
Let $\si<\tau$ in $\I$. Then the label $(i,j)$ on any cover in $[\si,\tau]$ satisfies $i \geq \di(\si,\tau)$.
\end{Lemma}

\begin{proof}
Suppose $i<\di(\si,\tau)$ for the label $(i,j)$ on $\si \cov \pi \leq \tau$. Then $\pi(k)=\tau(k)$ for $k<i$ and $\tau(i)=\si(i)$. However, it follows from \Def~\ref{ct} that $\pi(i)>\si(i)$. Hence, $\pi[i,\si(i)+1]>\tau[i,\si(i)+1]$. By \Def~\ref{order}, this contradicts the fact that $\pi \leq \tau$. Thus $i \geq \di(\si,\tau)$. The result follows by induction.
\end{proof}

\begin{Lemma} \label{Fm}
Let $\si<\tau$ in $\Fm$, and let $C_I=\si \cov \si_1 \cov \cd \cov \si_k \cov \tau$ be the increasing $\si$-$\tau$-chain in $\I$ and $C_D=\si \cov \tau_k \cov \cd \cov \tau_1 \cov \tau$ the decreasing $\si$-$\tau$-chain in $\I$. Then $\{\si_1,\tau_1\} \cap \Fm \neq \0$.
\end{Lemma}

\begin{proof}
Assume to the contrary that neither $\si_1$ nor $\tau_1$ belongs to $\Fs$. Let $h=\di(\si,\tau)$, and let $(i_\si,j_\si)$ and $(i_\tau,j_\tau)$ be the labels on $\si \cov \si_1$ and $\tau_1 \cov \tau$, respectively. By Lemma~\ref{ineq}, $i_\si,i_\tau \geq h$. Since $\si(h) \neq \tau(h)$, it follows that $h$ is in some label on $C_I$ and some label on $C_D$. Since $C_I$ is increasing, $i_\si=h$, and since $\si_1 \notin \Fm$, $h$ is an exceedance of $\si$ (Type~5). Since $C_D$ is decreasing, $i_\tau=h$, and since $\tau_1 \notin \Fm$, $h$ is a fixed point of $\tau_1$ (Type~1). Hence, $\si[h,h+1]>\tau_1[h,h+1]$. By \Def~\ref{order}, this contradicts the fact that $\si \leq \tau_1$.
\end{proof}

\begin{Lemma} \label{Fs}
Let $\si<\tau$ in $\Fs$, and let $C_I=\si \cov \si_1 \cov \cd \cov \si_k \cov \tau$ be the increasing $\si$-$\tau$-chain in $\I$ and $C_D=\si \cov \tau_k \cov \cd \cov \tau_1 \cov \tau$ the decreasing $\si$-$\tau$-chain in $\I$. Then $\{\si_1,\tau_1\} \cap \Fs \neq \0$.
\end{Lemma}

\begin{proof}
Assume to the contrary that neither $\si_1$ nor $\tau_1$ belongs to $\Fs$. Let $h=\di(\si,\tau)$, and let $(i_\si,j_\si)$ and $(i_\tau,j_\tau)$ be the labels on $\si \cov \si_1$ and $\tau_1 \cov \tau$, respectively. By Lemma~\ref{ineq}, $i_\si,i_\tau \geq h$. Since $\si(h) \neq \tau(h)$, it follows that $h$ is in some label on $C_I$ and some label on $C_D$. Since $C_I$ is increasing, $i_\si=h$, and since $\si_1 \notin \Fs$, $h$ is a fixed point of $\si$ (Type~1). Since $C_D$ is decreasing, $i_\tau=h$, and since $\tau_1 \notin \Fs$, $h$ is an exceedance of $\tau_1$ (Type~5).

Let $m$ be such that $h$ is an exceedance of $\tau_1,\p,\tau_{m-1}$ and a fixed point of $\tau_m$ (with $\tau_{k+1}=\si$). Then the labels on $\tau \voc \tau_1 \voc \cd \voc \tau_m$ are $(h,j_1),\p,(h,j_m)$, where $j_1<j_2<\cdots<j_m$. Since $\tau_1>\tau_2>\cdots>\tau_{m-1}$, $\tau_1(h)>\tau_2(h)>\cdots>\tau_{m-1}(h)$. Since $h$ is a fixed point of $\tau_m$ but an exceedance of $\tau_{m-1}$, the cover $\tau_m \cov \tau_{m-1}$ is of Type~1 or 2, whence $\tau_{m-1}(h)=j_m$ or $\tau_{m-1}(h)=\tau_m(j_m)>j_m$, respectively; hence, $\tau_{m-1}(h) \geq j_m$. Therefore, $j_1<j_m \leq \tau_{m-1}(h)<\tau_1(h)$. However, since the cover $\tau_1 \cov \tau$ is of Type~5, $\tau_1(h)<j_1$, which is a contradiction.
\end{proof}

\begin{Lemma} \label{Fi}
Let $\si<\tau$ in $\Fi$, and let $C_I=\si \cov \si_1 \cov \cd \cov \si_k \cov \tau$ be the increasing $\si$-$\tau$-chain in $\I$ and $C_D=\si \cov \tau_k \cov \cd \cov \tau_1 \cov \tau$ the decreasing $\si$-$\tau$-chain in $\I$. Then $\{\si_1,\tau_1\} \cap \Fi \neq \0$.
\end{Lemma}

\begin{proof}
Assume to the contrary that neither $\si_1$ nor $\tau_1$ belongs to $\Fi$. Then $\si$ and $\tau$ have $a_2$ fixed points while $\si_1$ and $\tau_1$ have $a_2+2$ fixed points, $\si$ and $\tau$ have $a_1$ fixed points while $\si_1$ and $\tau_1$ have $a_1-2$ fixed points, or $\si$ has $a_2$ fixed points and $\tau$ has $a_1$ fixed points (or vice versa) while $\si_1$ has $a_2+2$ fixed points and $\tau_1$ has $a_1-2$ fixed points.

In the first case, we get a contradiction as in the proof of Lemma~\ref{Fm}, and in the second case, we get a contradiction as in the proof of Lemma~\ref{Fs}. In the third case, we get a contradiction because if $\pi \cov \pi'$, then the number of fixed points of $\pi'$ is the same as, two more than, or two less than the number of fixed points of $\pi$.
\end{proof}

\begin{Prop} \label{FmFsFigraded}
The posets $\Fm$, $\Fs$, and $\Fi$ are graded.
\end{Prop}

\begin{proof}
This follows from Lemmas~\ref{graded}, \ref{Fm}, \ref{Fs}, and \ref{Fi}.
\end{proof}

In the following two results, we consider the sets $A$ for which $\FA$ is not graded.

\begin{Prop} \label{-o-}
If there is an $i \in [2,n-4]$ such that $i \in A$ but $i-2,i+2 \notin A$, then $\FA$ is not graded.
\end{Prop}

\begin{proof}
We first show that $\Fo_6^2$ is not graded. Let $\si=124365$ and $\tau=426153$, and consider the interval $[\si,\tau]$. Both $C_1=\si \cov 143265 \cov 423165 \cov \tau$ and $C_2=\si \cov 126453 \cov 216453 \cov \tau$ are $\si$-$\tau$-chains in $\Io_6$. However, $(\si,216453) \cap \Fo_6^2=\0$, whence $C_1$ is a $\si$-$\tau$-chain in $\Fo_6^2$ of length $3$, while $C_2-\{126453\}$ is a $\si$-$\tau$-chain in $\Fo_6^2$ of length $2$. Thus $\Fo_6^2$ is not graded.

Now we have to obtain the right number of fixed points. To achieve this, concatenate each of the involutions above with the sequence
\[
78 \cd (i+4)(i+6)(i+5) \cd n(n-1). \qedhere
\]
\end{proof}

\begin{Prop} \label{o-o}
If there is an $i \notin A$ and a positive integer $m$ such that $i-2,i+2m \in A-\{n\}$, then $\FA$ is not graded.
\end{Prop}

\begin{proof}
We first prove that $\Fo_k^{\{0,k-2\}}$, where $k \geq 6$ is even, is not graded. Let $\si=12 \cdots (k-2)k(k-1)$ and $\tau=k23 \cdots (k-1)1$, and consider the interval $[\si,\tau]$. We obtain a $\si$-$\tau$-chain $C$ in $\Io_k$ by $k-2$ $fe$-rises with labels $(k-2,k-1),(k-3,k-2),\p,(1,2)$ (from $\si$ to $\tau$). We also obtain a $\si$-$\tau$-chain in $\I$ by $(k-2)/2$ $ff$-rises with labels $(1,2),(3,4),\p,(k-3,k-2)$, followed by $(k-2)/2$ crossing $ee$-rises with labels $(k-3,k-1),(k-5,k-3),\p,(1,3)$.

Let $\pi$ be the fixed point-free involution obtained after the $ff$-rises. Since each $ff$-rise decreases the number of fixed points and $\Io_k$ is graded, $(\si,\pi) \cap \Fo_k^{\{0,k-2\}}=\0$, and since each crossing $ee$-rise increases the number of fixed points and $\Io_k$ is graded, $(\pi,\tau) \cap \Fo_k^{\{0,k-2\}}=\0$. Hence, $C$ is a $\si$-$\tau$-chain in $\Fo_k^{\{0,k-2\}}$ of length $k-2$, while $\si \cov \pi \cov \tau$ is a $\si$-$\tau$-chain in $\Fo_k^{\{0,k-2\}}$ of length $2$. Thus $\Fo_k^{\{0,k-2\}}$ is not graded. Figure~\ref{small} illustrates the situation when $k=6$.

Now we have to obtain the right number of fixed points. Assume, \WLOGtwo, that $m$ is minimal, let $k=2m+4$, and concatenate each of the involutions above with the sequence
\[
(k+1)(k+2) \cd (k+i-2)(k+i)(k+i-1) \cd n(n-1). \qedhere
\]
\end{proof}

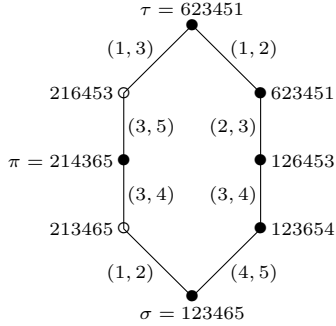
\begin{figure}[htb]
\begin{center}
\begin{tikzpicture}[scale=0.9]
  \node at (0,0) {$\bullet$}; \node at (0,4) {$\bullet$};
  \node at (1,1) {$\bullet$}; \node at (1,2) {$\bullet$}; \node at (1,3) {$\bullet$};
  \node at (-1,1) {$\circ$}; \node at (-1,2) {$\bullet$}; \node at (-1,3) {$\circ$};
  \node [below] at (0,0) {\tiny $\si=123465$}; \node [above] at (0,4) {\tiny $\tau=623451$};
  \node [right] at (1,1) {\tiny $123654$}; \node [right] at (1,2) {\tiny $126453$}; \node [right] at (1,3) {\tiny $623451$};
  \node [left] at (-1,1) {\tiny $213465$}; \node [left] at (-1,2) {\tiny $\pi=214365$}; \node [left] at (-1,3) {\tiny $216453$};
  \node at (0.9,0.35) {\tiny $(4,5)$}; \node at (0.6,1.5) {\tiny $(3,4)$}; \node at (0.6,2.5) {\tiny $(2,3)$}; \node at (0.9,3.65) {\tiny $(1,2)$};
  \node at (-0.9,0.35) {\tiny $(1,2)$}; \node at (-0.6,1.5) {\tiny $(3,4)$}; \node at (-0.6,2.5) {\tiny $(3,5)$}; \node at (-0.9,3.65) {\tiny $(1,3)$};
  \draw (0,0)--(1,1)--(1,2)--(1,3)--(0,4)--(-1,3)--(-1,2)--(-1,1)--(0,0);
\end{tikzpicture}
\end{center}
\caption{Two $\si$-$\tau$-chains in $\Io_6$ of length $4$, and two $\si$-$\tau$-chains in $\Fo_6^{\{0,4\}}$ of length $4$ (right) and length $2$ (left); the involutions marked by a $\bullet$ belong to $\Fo_6^{\{0,4\}}$, and the involutions marked by a $\circ$ belong to $\Io_6-\Fo_6^{\{0,4\}}$. On the edges (covers in $\Io_6$) are the labels $(i,j)$.} \label{small}
\end{figure}

We are now ready to prove the main result of this paper:

\begin{proof}[Proof of \Th~\ref{MT}]
The first claim follows from Lemma~\ref{bij} and \Pros~\ref{FmFsFigraded}, \ref{-o-}, and \ref{o-o}. (It is readily checked that if $\Fn$ does not belong to $\{\0,\Fm,\Fs,\Fi\}$, then either there is an $i \in [2,n-4]$ such that $i \in A$ but $i-2,i+2 \notin A$, or there are an $i \notin A$ and a positive integer $m$ such that $i-2,i+2m \in A-\{n\}$.) If $n \notin A$, then the second claim follows as in the proof of Lemma~\ref{graded}, together with Lemma~\ref{rank}. If $n \in A$, then each element's rank increases by 1. The third claim follows from the second claim and \Pro~\ref{top}.
\end{proof}


\section{EL-shellability of $\F^0$} \label{EL}

In this section, we give a new proof of \Th~\ref{C-C-T}, due to Can, Cherniavsky, and Twelbeck~\cite{C-C-T}. Our proof is largely based on the same main idea as their proof, together with the technique used in the proof of Lemma~\ref{Fm}. The proof in \cite{C-C-T} goes as follows:

Let $\si<\tau$ in $\F^0$. It follows from, e.g., \Th~\ref{Hultman} and the paragraphs following it, that there exists a $\si$-$\tau$-chain in $\I$ that is contained in $\F^0$. Let $C$ be the lex-maximal such chain. The idea of the proof is to show that $C$ is decreasing. Then, by reversing the lexicographic order on the set $\{(i,j) \in [n]^2 \mid i<j\}$ (i.e., by letting $(i_1,j_1)<(i_2,j_2)$ if and only if $i_1>i_2$, or $i_1=i_2$ and $j_1>j_2$), one obtains an edge-labelling of $\F^0$ such that in each interval, there is an increasing $\si$-$\tau$-chain which is lex-minimal. By Lemma~\ref{decr} and the remark following it, this is an EL-labelling of $\F^0$.

We use the same main idea, namely, to show that the decreasing $\si$-$\tau$-chain in $\I$ is contained in $\F^0$, and then reverse the lexicographic order. However, we give a direct proof of this fact. By using the same technique as in the proof of Lemma~\ref{Fm}, we get a very short argument.

\begin{Lemma} \label{F0}
Let $\si<\tau$ in $\F^0$ and let $C_D=\si \cov \tau_k \cov \cd \cov \tau_1 \cov \tau$ be the decreasing $\si$-$\tau$-chain in $\I$. Then $\tau_1,\p,\tau_k \in \F^0$.
\end{Lemma}

\begin{proof}
Since the decreasing $\si$-$\tau_1$-chain in $\I$ is $\si \cov \tau_k \cov \cd \cov \tau_2 \cov \tau_1$, it suffices to prove that $\tau_1 \in \F^0$.

Let $h=\di(\si,\tau)$, let $C_I=\si \cov \si_1 \cov \cd \cov \si_k \cov \tau$ be the increasing $\si$-$\tau$-chain in $\I$, and let $(i_\si,j_\si)$ and $(i_\tau,j_\tau)$ be the labels on $\si \cov \si_1$ and $\tau_1 \cov \tau$, respectively. By Lemma~\ref{ineq}, $i_\si,i_\tau \geq h$. Since $\si(h) \neq \tau(h)$, it follows that $h$ is in some label on $C_I$ and some label on $C_D$. Since $C_I$ is increasing, $i_\si=h$, and since $\si$ has no fixed points, $h$ is an exceedance of $\si$ (Type~4, 5, or 6). Since $C_D$ is decreasing, $i_\tau=h$, and were $\tau_1 \notin \F^0$, $h$ would be a fixed point of $\tau_1$ (Type~1). Hence, by \Def~\ref{order}, $\tau_1 \in \F^0$.
\end{proof}

We can now complete the proof of \Th~\ref{C-C-T}:

\begin{proof}[Proof of \Th~\ref{C-C-T}]
Let $\si<\tau$ in $\F^0$. By Lemma~\ref{F0}, the decreasing $\si$-$\tau$-chain in $\I$ is contained in $\F^0$. If we can show that this chain is lex-maximal, then by reversing the lexicographic order and invoking Lemma~\ref{decr}, we are done.

In order to obtain a contradiction, let $C=\si_1 \cov \cd \cov \si_k$ be the lex-maximal $\si$-$\tau$-chain in $\I$, and assume that it is not decreasing; say that $\la(\si_1,\si_2) \leq \la(\si_2,\si_3)$. By Lemma~\ref{incr}, $\si_1 \cov \si_2 \cov \si_3$ is lex-minimal among the $\si_1$-$\si_3$-chains in $\I$. Hence, $\si_1 \cov \si_2' \cov \si_3 \cov \cd \cov \si_k$, where $\si_1 \cov \si_2' \cov \si_3$ is the decreasing $\si_1$-$\si_3$-chain, is lex-larger than $C$, which is a contradiction.
\end{proof}

Is it possible to use the same idea to prove that every interval in $\FA$ is EL-shellable for some $A \neq \{0\}$? Unfortunately, the answer is no, since for all $A \neq \{0\}$ (except the trivial case $A=\{n\}$ and the case when $\FA=\I$), it is possible to find $\si_1<\tau_1$ and $\si_2<\tau_2$ in $\FA$, such that the increasing $\si_1$-$\tau_1$-chain and the decreasing $\si_2$-$\tau_2$-chain in $\I$, are of length 2 and are not contained in $\FA$.

\section*{Acknowledgements}

The author thanks Axel Hultman for helpful comments and fruitful discussions.



\begin{thebibliography}{99}
  \bibitem {Bjorner}  A.~Bj\"orner, Shellable and Cohen-Macaulay partially ordered sets, \Trams{260}{1980}, 159--183.
  \bibitem {B-B}      A.~Bj\"orner and F.~Brenti, \emph{Combinatorics of Coxeter groups}, \Spr{231}{2005}.
  \bibitem {C-C-T}    M.~B.~Can, Y.~Cherniavsky, and T.~Twelbeck, Lexicographic shellability of the Bruhat-Chevalley order on fixed-point-free involutions, \emph{Israel J. Math.}, to appear.
  \bibitem {Edelman}  P.~H.~Edelman, The Bruhat order of the symmetric group is lexicographically shellable, \Proc{82}{1981}, 355--358.
  \bibitem {Helgason} S.~Helgason, \emph{Differential geometry, Lie groups, and symmetric spaces}, \Pam{80}{1978}.
  \bibitem {Hultman1} A.~Hultman, Fixed points of involutive automorphisms of the Bruhat order, \Adv{195}{2005}, 283--296.
  \bibitem {Hultman2} A.~Hultman, The combinatorics of twisted involutions in Coxeter groups, \Trams{359}{2007}, 2787--2798.
  \bibitem {Hultman3} A.~Hultman, Twisted identities in Coxeter groups, \Algcomb{28}{2008}, 313--332.
  \bibitem {Incitti}  F.~Incitti, The Bruhat order on the involutions of the symmetric group, \Algcomb{20}{2004}, 243--261.
  \bibitem {R-S1}     R.~W.~Richardson and T.~A.~Springer, The Bruhat order on symmetric varieties, \Geom{35}{1990}, 389--436.
  \bibitem {R-S2}     R.~W.~Richardson and T.~A.~Springer, Complements to: The Bruhat order on symmetric varieties, \Geom{49}{1994}, 231--238.
  \bibitem {Stanley}  R.~P.~Stanley, \emph{Enumerative combinatorics. Vol.~1}, \Cam{49}{1997}.
\end{thebibliography}

\end{document}